\documentclass[a4paper,12pt]{amsart}

\usepackage{bbm}
\usepackage{units}
\usepackage{mathtools}
\usepackage{cancel}
\usepackage{amsmath}
\usepackage{amsfonts}
\usepackage{latexsym}
\usepackage{amssymb}
\usepackage{mathrsfs}
\usepackage{amscd}
\usepackage{hyperref}
\usepackage{soul}
\usepackage{psfrag}
\usepackage{graphicx}
\usepackage{ulem}
\usepackage{fullpage}
\usepackage[all]{xy}
\usepackage{rotating}
\usepackage{url}
\usepackage{color}
\usepackage{enumerate}
\usepackage{cleveref}

\hypersetup{
	colorlinks=true, 
	linkcolor=red, 
	citecolor=blue,
}

\usepackage{dsfont}
\usepackage{tikz}
\usetikzlibrary{arrows,backgrounds,arrows.meta,decorations.markings}
\usepackage{tikz-cd}

\numberwithin{equation}{section}
\numberwithin{figure}{section}

\theoremstyle{plain}
\newtheorem{thm}{Theorem}[section]

\theoremstyle{plain}
\newtheorem{lem}[thm]{Lemma}

\theoremstyle{remark}
\newtheorem{rem}[thm]{Remark}

\theoremstyle{plain}
\newtheorem{cor}[thm]{Corollary}

\theoremstyle{definition}
\newtheorem{defn}[thm]{Definition}

\theoremstyle{definition}

\theoremstyle{definition}

\theoremstyle{plain}
\newtheorem{prop}[thm]{Proposition}

\theoremstyle{plain}

\theoremstyle{definition}

\theoremstyle{plain}

\newcommand{\comments}[1]{}

\newcommand{\rab}{\rangle}

\newcommand{\lab}{\langle}

\newcommand{\mcal}{\mathcal}

\newcommand{\N}{\mathbb N}

\newcommand{\C}{\mathbb{C}}

\newcommand{\vlon}{\varepsilon}

\title{Peripheral Poisson boundary on full Fock space}
\author{Mainak Ghosh}

\begin{document}
	\maketitle
	\global\long\def\vlon{\varepsilon}
	\global\long\def\bt{\bowtie}
	\global\long\def\ul#1{\underline{#1}}
	\global\long\def\ol#1{\overline{#1}}
	\global\long\def\norm#1{\left\|{#1}\right\|}
	\global\long\def\os#1#2{\overset{#1}{#2}}
	\global\long\def\us#1#2{\underset{#1}{#2}}
	\global\long\def\ous#1#2#3{\overset{#1}{\underset{#3}{#2}}}
	\global\long\def\t#1{\text{#1}}
	\global\long\def\lrsuf#1#2#3{\vphantom{#2}_{#1}^{\vphantom{#3}}#2^{#3}}
	\global\long\def\tr{\triangleright}
	\global\long\def\tl{\triangleleft}
	\global\long\def\cc90#1{\begin{sideways}#1\end{sideways}}
	\global\long\def\turnne#1{\begin{turn}{45}{#1}\end{turn}}
	\global\long\def\turnnw#1{\begin{turn}{135}{#1}\end{turn}}
	\global\long\def\turnse#1{\begin{turn}{-45}{#1}\end{turn}}
	\global\long\def\turnsw#1{\begin{turn}{-135}{#1}\end{turn}}
	\global\long\def\fusion#1#2#3{#1 \os{\textstyle{#2}}{\otimes} #3}
	
	\global\long\def\abs#1{\left|{#1}\right	|}
	\global\long\def\red#1{\textcolor{red}{#1}}
 
 \begin{abstract}
 	The operator space generated by peripheral eigenvectors of a unital normal completely positive map $P$ on a von Neumann algebra has a C*-algebra structure. This C*-algebra is known as the \textit{peripheral Poisson boundary} of $P$. For a separable Hilbert space $H$, consider the full fock space defined over $H$. In this paper, we study the peripheral Poisson boundary of the completely positive map, induced by left creation operators of the basis vectors of $H$, on $B(\mcal F(H))$ and explore its behavior with respect to the Poisson boundary.
 \end{abstract}

\section{Introduction}

The notion of noncommutative poisson boundaries for normal unital completely positive (ucp) maps was introduced by M. Izumi \cite{Izm02}, extending the notion of Poisson boundary theory for random walks on groups. M. Izumi also studied the structure of higher relative commutants of the core inclusions of finite index subfactors of $\t{II}_1$ factors as noncommutative Poisson boundaries of ucp maps on finite type $\t{I}$ von Neumann algebras \cite{Izm04}.
 
Suppose $N$ is a von Neumann algebra and $L$ a self-adjoint linear subspace of $N$ containing the identity (that is, $L$ is a weakly closed operator system in $N$). By virtue of \cite[Theorem 3.1]{CE}, a completely positive projection $E : N \to L $ induces a von Neumann algebra structure on $L$ with respect to the `Choi-Effros' product,
\[ x \circ y = E(xy) \ \ \t{for all} \ \ x, y \in L . \]

Let $P : N \to N$ be a normal ucp map. The operator space of fixed points of $P$ becomes a von Neumann algebra using the `Choi-Effros' product. This von Neumann algebra (or its concrete realization) is known as the \textit{noncommutative Poisson boundary} of $\left(P, N\right)$. Recently, this concept has gained considerable attention in various contexts \cite{Izm04, Izm12, BKT, BBDR, DGGJ}.
\vspace*{2mm}

A spectral approach of decomposing the dynamics of open quantum systems into persistent and transient parts was sought in \cite{BKT} in the Heisenberg picture.
In \cite{BKT}, the operator space formed by eigenvectors corresponding to eigenvalues on the unit circle for ucp maps were investigated. The authors imposed a new product, using dilation theory, on this operator space  to obtain a C*-algebra. This C*-algebra is termed as the \textit{peripheral Poisson boundary}. The product on the peripheral Poisson boundary is an extension of the `Choi-Effros' product on the Poisson boundary. Surprisingly, peripheral Poisson boundary may not always be a von Neumann algebra (see \cite{BKT}).

\vspace*{2mm}

Suppose $H$ is a separable Hilbert space with an orthonormal basis $\left\{e_i : i \in \Theta \right\}$. Consider the full fock space $\mcal F(H)$ defined over $H$ and a sequence of positive real numbers $\left\{\omega_i : i \in \Theta \right\}$ such that, $\us{i \in \Theta}{\sum} \omega_i = 1 $. Define the normal ucp $P_\omega$ on $B(\mcal F(H))$ as,
\[ P_\omega (x) = \us{i \in \Theta}{\sum} \omega_i l_{e_i}^* x l_{e_i} \ \ \t{for all} \ \ x \in B(\mcal F(H)) . \]
The Poisson boundary for the ucp map $P_\omega$ was studied in \cite{BBDR}. Our goal in this paper is to study the peripheral Poisson boundary associated to $P_\omega$, and to examine the relationship between peripheral Poisson boundary and the Poisson boundary of $P_\omega$. 

We will see in \Cref{ppb}, that peripheral Poisson boundary is independent (upto *-isomorphism) of the choice of orthonormal basis of $H$, and that it strictly contains the Poisson boundary as a C*-subalgebra. We then move on to study the behavior of the Poisson boundary with respect to the peripheral Poisson boundary. The following is the main theorem of the paper.

\subsection*{Theorem A}   
\textit{The relative commutant of the poisson boundary with respect to the peripheral Poisson boundary is trivial}.

\vspace*{2mm}
As a corollary, we obtain the following result.
\subsection*{Corollary A}
\textit{The center of the peripheral Poisson boundary is trivial}.

\vspace*{2mm}
Next we provide a description of product in the peripheral Poisson boundary and move on to address the following question :

\subsection*{Question 1}
\textit{Is there a conditional expectation from the peripheral Poisson boundary onto the Poisson boundary ?}

\vspace*{2mm}

We answer the above question in the affirmative.
In \cite{BKT}, for any normal ucp $\tau : M \to M $, each eigenspace (corresponding to eigenvalues in the unit circle) has been equipped with natural left and right actions of the \textit{Poisson boundary}, and with an inner product taking values in the Poisson boundary, making it a Hilbert C*-bimodule over the Poisson boundary. This approach is not going to work in general if we want to make the whole peripheral Poisson boundary a pre-Hilbert C*-bimodule over the Poisson boundary. We handle this in \Cref{ppb}, by making use of the conditional expectation from the peripheral Poisson boundary onto the Poisson boundary, and prove the following theorem.

\subsection*{Theorem B}
\textit{The peripheral Poisson boundary becomes a pre-Hilbert C*-bimodule over the Poisson boundary}.
 
\vspace*{2mm}
We briefly discuss the contents of this article.

In \Cref{prelim}, we will quickly go through some basic results and definitions, and set up notations, which will be used later.

\Cref{ppb} begins by establishing the fact that the peripheral Poisson boundary is independent of the choice of the orthonormal basis of the Hilbert space $H$. Further, we show strict inclusion of the Poisson boundary as a C*-subalgebra of the peripheral Poisson boundary. Then we compute several important formulas and use them to prove \Cref{centerthm}. Then we move on to describe the product in the peripheral Poisson boundary. Finally, we build a conditional expectation from the peripheral Poisson boundary onto the Poisson boundary.

\subsection*{Acknowledgements}
The author would like to thank Prof. B. V. R. Bhat and Dr. Narayan Rakshit for several fruitful discussions at various stages of the project. A part of the project was completed when the author visited Indian Statistical Institute, Bangalore Centre. The author would like to thank the Stat-Math unit for the warm hospitality. The author gratefully acknowledges the support received from Prof. Bhat (J C Bose Fellowship No. JBR/2021/000024) for funding his travel expenses.

\section{Preliminaries}\label{prelim}

In this section we recall various standard facts concerning full Fock spaces over Hilbert spaces and the Peripheral Poisson boundary of a ucp map, pertaining to our results. We refer the reader to \cite{BKT, BBDR}, and the references therein for a more detailed study. Throughout our discussion $H$ will denote a separable Hilbert space with an orthonormal basis $\left\{e_i : i \in \Theta \right\}$, where $\Theta$ stands for the set $\left\{1,2, \cdots, n \right\}$ ($n \in \N$, $n > 1 $) or the set $\N$.

\subsection{Full fock spaces over a Hilbert space}\

With notations as above, consider the full fock space over $H$ defined by, 
\[ \mcal F(H) = \us{n \geq 0}{\oplus} H^{\otimes n}\]
where $H^{\otimes 0} \coloneqq \C \Omega$ and for $n \geq 1$, $H^{\otimes n}$ is the (Hilbert space) tensor product of $n$-copies of $H$. Here $\Omega$ is a fixed complex number with modulus $ 1 $ and we refer it as vacuum vector.

Let $\Lambda$ and $\Lambda^*$ denote the sets,
\[ \Lambda = \us{n \geq 0}{\bigcup} \Theta^n \ \ \t{and} \ \ \Lambda^* = \us{n \geq 1}{\bigcup} \Theta^n , \]
where for $n \geq 1$, $\Theta^n$ denotes the $n$-fold Cartesian product of $\Theta$ and $\Theta^0 = \left\{ () \right\}$, where $()$ is the empty tuple. The elements of $\Theta^n$, $n \geq 1$, are referred to as sequences of length $n$. If $I$ is a sequence of length $n$, $n \geq 1$, say $I = (i_1, i_2, \cdots, i_n)$, we shall, in the sequel, simply write $i_1 i_2 \cdots i_n $ for $I$. We will treat the empty tuple as a sequence of length of $0$. If $I$ is a sequence in $\Lambda$, we denote its length by $\abs{I}$. For $I= ()$, the empty tuple, we set \[ e_I \coloneqq \Omega , \]
and for $I \in \Lambda^*$, say $I = i_1 i_2 \cdots i_k$, we define \[ e_I \coloneqq e_{i_1} \otimes e_{i_2} \otimes \cdots \otimes e_{i_k}\]

With these notations, $\mcal B \coloneqq \left\{ e_I : I \in \Lambda \right\}$ forms an orthonormal basis for $\mcal F(H)$. For $I \in \Lambda^*$, we will call $I^{op}$ the sequence which is reverse to that of $I$, that is, if
\[ I = i_1 i_2 \cdots i_k, \ \ \t{then} \ \ I^{op} = i_k \cdots i_2 i_1 .  \]

Let $I = i_1 i_2 \cdots i_k \in \Theta^k$, $k \geq 0$. For any non-negative integer $m \leq k$, we denote by $I_m$ the subsequence of $I$ of length $m$ defined by \[ I_m = \begin{cases} 
       \t{empty tuple}, \ \ \t{if} \ \ m = 0 \\
       I_m = i_1 i_2 \cdots i_m \ \ \t{if} \ \ 1 \leq m \leq k        \end{cases}\]
       
If $I, J$ are two sequences in $\Lambda$, then $IJ$ will denote the sequence of length $\abs{I} + \abs{J}$ obtained by juxtaposition. That is, if $I = i_1 i_2 \cdots i_k$, and $J = j_1 j_2 \cdots j_l$, then \[ IJ = i_1 \cdots i_k j_1 \cdots j_l \] 

Further, given a sequence $I \in \Lambda$ and $n \in \N$, we denote by $I^n$ to be the sequence $\underbrace{I \cdots I}_\t{$n$ times} $.

\begin{defn}For each $\xi \in H$, 	
	\begin{itemize}
		\item [(i)] the \textit{left creation operator} associated with $\xi$, denoted by $l_\xi$, is the bounded operator on $\mcal F(H)$ that satisfies, \[ l_\xi (\eta) = \xi \otimes \eta \ \ \t{for all} \ \ \eta \in \mcal F(H) .  \]     
\item [(ii)]  the \textit{right creation operator} associated with $\xi$, denoted by $r_\xi$, is the bounded operator on $\mcal F(H)$ that satisfies, \[ r_\xi (\eta) =  \eta \otimes \xi \ \ \t{for all} \ \ \eta \in \mcal F(H) . \]
	\end{itemize}
\end{defn}

The adjoint of $l_\xi^*$ (respectively, $r_\xi^*$) of $l_\xi$ (respectively, $r_\xi$) is called the left (respectively, right) annihilation operator associated with $\xi$. It is easy to verify that $l_\xi^*$ and $r_\xi^*$ satisfy the relations:
\[ l_\xi^* (\Omega) = 0, \ \ \t{and} \ \ l_\xi^* (\eta_1 \otimes \eta_2 \otimes \cdots \otimes \eta_k) = \lab \eta_1, \xi \rab \eta_2 \otimes \cdots \otimes \eta_k \ \ \t{for} \ k \in \N, \eta_i \in H . \]
\[ r_\xi^* (\Omega) = 0, \ \ \t{and} \ \ r_\xi^* (\eta_1 \otimes \eta_2 \otimes \cdots \otimes \eta_k) = \lab \eta_k, \xi \rab \eta_1 \otimes \cdots \otimes \eta_{k-1} \ \ \t{for} \ k \in \N, \eta_i \in H .  \]

\subsection*{Notations}
For $i \in \Theta$, we simply use the notation $r_i$ (respectively, $l_i$) to denote $r_{e_i}$ (respectively, $l_{e_i}$). For $I \in \Lambda$, we define $l_I$ and $r_I$ as follows,
\[ l_I =  r_I = 1 \ \ \t{if $I$ is the empty tuple} \]
\[l_I = l_{i_1} l_{i_2} \cdots l_{i_k} \ \ \t{and} \ \ r_I = r_{i_1} r_{i_2} \cdots r_{i_k} \ \ \t{if $I = i_1 i_2 \cdots i_k$} . \]

We record the following result which will be useful later for our purposes.

\begin{lem}\
	
	\begin{itemize}
		\item [(i)] $l_i^* p_\Omega = r_i^* p_\Omega = p_\Omega l_i = p_\Omega r_i = 0 $ for all $i \in \Theta$, where $p_\Omega$ denotes the orthogonal projection of $\mcal F(H)$ onto $\C \Omega$
		
		\item [(ii)] For any $\xi, \eta \in H$,
		\[ r_\xi^* r_\eta = l_\xi^* l_\eta = \lab \eta, \xi \rab 1, r_\xi^* l_\eta = l_\eta r_\xi^* + \lab \eta, \xi \rab p_\Omega \ \ \t{and} \ \ l_\eta^* r_\xi = r_\xi l_\eta^* + \lab \xi, \eta \rab p_\Omega \]
		In particular,
		\[ r_i^* r_j = l_i^* l_j = \delta_{i,j}, r_i^* l_j = l_j r_i^* + \delta_{i,j} p_\Omega \ \ \t{and} \ \ l_i^* r_j = r_j l_i^* + \delta_{i,j} p_\Omega \ \ \t{for all} \ \ i,j \in \Theta \]
	\end{itemize}
\end{lem}

\subsection{Peripheral Poisson boundary}\

Suppose $\mcal A$ is a von Neumann algebra and $\tau : \mcal A \to \mcal A $ is a normal ucp map. The fixed point space of $\tau$ (also known as the noncommutative Poisson boundary of $\tau$), \[ F(\tau) \coloneqq \left\{ x \in \mcal A : \tau (x)= x \right\} . \]
forms a von Neumann algebra with respect to the `Choi-Effros' product `$\circ$' \cite{Izm12}. Namely, for two elements $x, y \in F(\tau)$, the product is given by, \[ x \circ y = \t{SOT}-\us{n \to \infty}{\t{lim}} \tau^n (xy) . \]

\vspace*{2mm}

For a normal ucp map $\tau : \mcal A \to \mcal A $ and for $\lambda \in \mathbbm{T}$, define \[ E_\lambda (\tau) \coloneqq \left\{ x \in \mcal A : \tau(x) = \lambda x \right\} . \]
Take $E(\tau) \coloneqq \t{span}\left\{x \in \mcal A : x \in E_\lambda (\tau) \ \ \t{for some} \ \ \lambda \in \mathbbm{T} \right\}$. Define $P(\tau) = \ol{\t{span}} E(\tau)$. The operator space $P(\tau)$ is called the \textit{peripheral Poisson boundary} of $\tau$. In the next result, we observe that $P(\tau)$ becomes a C*-algebra with respect to a modified `Choi-Effros' product.

\begin{thm}\cite{BKT}
	Let $\mcal A \subseteq B(H)$ be a von Neumann algebra and let $\tau : \mcal A \to \mcal A$ be a normal unital completely positive map. Let $\left(K, \mcal B, \theta\right)$ be the minimal dilation of $\tau$. Let $T$ denote the compression map $z \mapsto pzp$ restricted to $P(\theta)$. Then the completely positive map $T$ maps the C*-algebra $P(\theta)$ isometrically and bijectively to $P(\tau)$. In particular, setting, \[ x \circ y \coloneqq T(T^{-1}(x) T^{-1}(y)) \ \ \t{for all} \ \ x, y \in P(\tau)  \] makes $(P(\tau), \circ)$ a unital C*-algebra. 
\end{thm}

\begin{cor}\cite{BKT}
	If $x \in E_\lambda(\tau)$ and $y \in E_\mu (\tau)$, then, 
	\[ x \circ y = \normalfont\t{SOT}-\us{n \to \infty}{\t{lim}} (\lambda \mu)^{-n} \tau^n (xy) .  \]
\end{cor}

\section{Peripheral Poisson boundary}\label{ppb}

Consider the pair $\left(B(\mcal F (H)), P_{\omega})\right)$, where $\omega = \left\{\omega_i : i \in \Theta \right\}$ such that $\us{i \in \Theta}{\sum}  \omega_i = 1 $ and $P_{\omega} : B(\mcal F(H)) \to B(\mcal F(H))$ is the normal ucp map defined as, $$P_{\omega}(x) = \sum_{i \in \Theta} \omega_i l_i^* x l_i \ \ \ \ \text{for all} \ \ x \in B(\mcal F(H)) . $$
In this section, we study some structural properties of the peripheral Poisson boundary $P(P_\omega)$ of $P_\omega$ and examine its behavior with the Poisson boundary $F(P_\omega)$. We begin by showing the peripheral Poisson boundary of $P_\omega$ is independent of the choice of orthonormal basis

\vspace*{2mm}

Let $\left\{ e_i : i \in \Theta \right\}$ and $\left\{ f_i : i \in \Theta \right\}$ be two orthonormal bases of $H$. Consider the corresponding ucp maps,
\[P_{\omega}(x) = \sum_{i \in \Theta} \omega_i l_{e_i}^* x l_{e_i} \ \ \t{and}  \ \ P'_{\omega}(x) = \sum_{i \in \Theta} \omega_i l_{f_i}^* x l_{f_i} \ \ \text{for all} \ \ x \in B(\mcal F(H)). \]

Our goal is to show that $P(P_\omega)$ and $P(P'_\omega)$ are isomorphic as C*-algebras

Consider the unitary $U : H \to H $  sending $e_i$ to $f_i$ for each $i \in \Theta$. The unitary $U$ induces the second quantization $\Gamma_U : \mcal F(H) \to \mcal F(H)$, which is an unitary on $\mcal F(H)$, defined by $\Gamma_U (\Omega) = \Omega$ and $\Gamma_U|_{H^{\otimes n}} = U^{\otimes n}$ for $n \in \N$. Clearly, $\Gamma_U$ induces the automorphism $\widetilde \Gamma_U$ of $B(\mcal F(H))$ given by,
\[ \widetilde \Gamma_U (x) = \Gamma_U x \Gamma_U^* \ \ \t{for} \ \ x \in B(\mcal F(H)) \]
It is easy to verify that $\widetilde \Gamma_U P_\omega = P'_\omega \widetilde \Gamma_U$. Using this identity it is easy to conclude that 
\[ P(P_\omega) \ni x \mapsto \widetilde \Gamma_U (x) \in P(P'_\omega) \] 
is a $*$-algebra isomorphism.

\subsection{Strict inclusion of Poisson boundary inside Peripheral Poisson boundary}

For each $\lambda \in \mathbbm{T} \setminus \left\{ 1 \right\}$, consider the maps, 
\[ \mcal F(H) \ni e_I \os{x_\lambda} \longmapsto \lambda^{\abs{I}} e_I \in \mcal F(H) \ \ \t{for} \ I \in \Lambda .   \]
Clearly, $x_\lambda$ is a unitary in $B(\mcal F(H))$ for each $\lambda \in \mathbbm{T} \setminus \left\{ 1 \right\}$. 

In the next lemma, we list some useful properties of $\left\{ x_\lambda : \lambda \in \mathbbm{T} \setminus \left\{ 1 \right\} \right\}$.

\begin{lem}\label{strictincl}\
	\begin{itemize}
		\item [(i)] $P_\omega (x_\lambda) = \lambda x_\lambda$, for $\lambda \in \mathbbm{T} \setminus \left\{1\right\}$.
		\item [(ii)] $x_\lambda \circ x_\mu = x_\lambda x_\mu = x_{\lambda \mu}$, for $\lambda, \mu \in \mathbbm{T} \setminus \left\{1\right\}$. 
	\end{itemize}
\end{lem}

\begin{proof}\
	\begin{itemize}
		\item [(i)] For each $I \in \Lambda$,
			\begin{align*}
			P_\omega(x_\lambda)(e_I) &= \us{i \in \Theta}{\sum} \omega_i l_i^* x_\lambda l_i (e_I)\\ &= \us{i \in \Theta}{\sum} \omega_i l_i^* x_\lambda  (e_{iI}) \\ &= \lambda^{\abs{I}+1} \us{i \in \Theta}{\sum} \omega_i l_i^* (e_{iI}) \\ &= \lambda x_\lambda (e_I) .
		\end{align*}
		
		\item [(ii)]  For $\lambda, \mu \in \mathbbm{T} \setminus \left\{1\right\}$, using (i), we get that, $P_\omega (x_\lambda x_\mu) = \left(\lambda \mu \right) x_\lambda x_\mu $. Hence, $P_\omega^n (x_\lambda x_\mu) = \left(\lambda \mu \right)^n x_\lambda x_\mu$. Thus, for each $I \in \Lambda$, we get, \begin{align*}
			\left(x_\lambda \circ x_\mu \right) (e_I) = \us{n \to \infty}{\t{lim}} \left(\lambda \mu \right)^{-n} P_\omega^n (x_\lambda x_\mu)(e_I) = \us{n \to \infty}{\t{lim}} \left(\lambda \mu \right)^{-n} \left(\lambda \mu \right)^n (x_\lambda x_\mu)(e_I) = x_\lambda x_\mu (e_I) . 
		\end{align*}  
	\end{itemize}
\end{proof}
By \Cref{strictincl}(i) we have, for each $\lambda \in \mathbbm{T} \setminus \left\{ 1 \right\}$, $x_\lambda \in P(P_\omega) \setminus F(P_\omega)$. Hence, the inclusion of $F(P_\omega)$ inside $P(P_\omega)$ is strict. 

\subsection{Center of the Peripheral Poisson boundary}\

In the next proposition we compute several multiplication formulae in $P(P_\omega)$ which will be used later.
 
\begin{prop}\label{formulaprop}
	Let $x \in P(P_\omega)$ and let $I, J \in \Lambda^*$. Then :
	\begin{itemize}
		\item [(i)] $ x \circ r_{I} = x r_I $.
		\item [(ii)] $r_I^* \circ x = r_I^*  x $.
		\item [(iii)] $r_J^* \circ x \circ r_I = r_J^* x r_I $.
		\item [(iv)] $r_I \circ x = r_I x + \displaystyle \sum_{t=1}^{\abs{I}} \omega_{\left(I^{op}\right)_t} r_{I_{\abs{I}-t}} p_\Omega x l_{\left(I^{op}\right)_t} $. In particular, $r_i \circ x = r_i x + \omega_i p_\Omega x l_i$, $i \in \Theta$.
		\item [(v)] $ x \circ r_{I}^* = x r_I^* + \displaystyle \sum_{t=1}^{\abs{I}} \omega_{\left(I^{op}\right)_t} l_{\left(I^{op}\right)_t}^* x p_\Omega r_{I_{\abs{I}-t}}^* $. In particular, $x \circ r_i^* = x r_i^* + \omega_i l_i^* x p_\Omega$, $i \in \Theta$.
	\end{itemize}
\end{prop}

\begin{proof}
	By \cite{BKT}, $P(P_\omega)$ forms a C*-algebra with respect to the generalized `Choi-Effros' product `$\circ$'. So, it is enough to prove the identities when $x$ is a peripheral eigenvector of $P_\omega$. Suppose we have  $P_\omega (x) = \lambda x$ for $\lambda \in \mathbbm{T} \setminus \left\{ 1 \right\}$.
	\begin{itemize}
		\item [(i)] It is easy to verify that $r_I l_i = l_i r_I$, for all $i \in \Theta, I \in \Lambda$. Now, \[ P_\omega (x r_I) = \us{i \in \Theta}{\sum} \omega_i l_i^* x r_I l_i = \us{i \in \Theta}{\sum} \omega_i l_i^* x l_i r_I  = P_\omega(x) r_I = \lambda x r_I . \]
		So, $x r_I \in E_\lambda (P_\omega)$. Thus, \[x \circ r_I = \t{SOT}- \us{n \to \infty}{\t{lim}} \lambda^{-n} P_\omega^n (x r_I) = x r_I \]
		
		\item [(ii)] Follows from part (i) of the proposition by taking adjoint.
		\item [(iii)] Follows from part (i) and part (ii) of the proposition.
		\item [(iv)] We follow the steps in \cite{BBDR} for the proof. We prove the result by induction on $\abs{I}$, the length of the sequence $I$. Let $\abs{I}= 1$, say, $I = i$ for some $i \in \Theta$. Then, 
		\begin{align*}
		P_\omega (r_i x) &= \us{j \in \Theta}{\sum} \omega_j l_j^* \left(r_i x \right) l_j  = \us{j \in \Theta}{\sum} \omega_j \left(r_i l_j^* + \delta_{i,j} p_\Omega \right) x l_j \\ &= r_i \left(\us{j \in \Theta}{\sum} \omega_j l_j^* x l_j \right) + \omega_i p_\Omega x l_i  \\ &= r_i P_\omega (x) + \omega_i p_\Omega x l_i = \lambda r_i x + \omega_i p_\Omega x l_i
		\end{align*}
		Therefore, \begin{align*}
			P_\omega^2 \left(r_i x\right) &= \lambda P_\omega(r_i x) + \omega_i P_\omega\left(p_\Omega x l_i \right) \\ &= \lambda P_\omega(r_i x) + \omega_i \us{j \in \Theta}{\sum} \omega_j l_j^* \left(p_\Omega x l_i \right) l_j \\ &= \lambda P_\omega \left(r_i x \right) \left(\t{Since} \  l_j^* p_\Omega = 0 \right).
		\end{align*}
		Consequently, $P_\omega^n \left(r_i x \right) = \lambda^n P_\omega \left(r_i x \right)$ for all $n \in \N$. Thus, \begin{equation}\label{rieqn}
			r_i \circ x =  \t{SOT}- \us{n \to \infty}{\t{lim}} \lambda^{-n} P_\omega^n (r_i x) = P_\omega (r_i x) = r_i x  + \omega_i p_\Omega x l_i .
		\end{equation} 
		Hence, the result is true for all $r_I$ with $\abs{I} = 1$. Suppose that the result is true for all sequences of length $m$, for some $m \in \N$. We show that the result is true for all sequences of length $m+1$. Let $I$ be a sequence of length $m+1$, say, $I = i_1 i_2 \cdots i_{m+1}$. By part (i) we get that $r_I = r_{i_1} \cdots r_{i_{m+1}} = r_{i_1} \circ \cdots \circ r_{i_{m+1}} $. Thus, $$ r_I \circ x = r_{i_1} \circ \left(r_{i_2} \circ \cdots \circ r_{i_{m+1}} \circ x \right) = r_{i_1} \circ \left(r_J \circ x \right) $$
		where $J = i_2 \cdots i_{m+1} $. Since $\abs{J} = m $, by induction hypothesis, we have $$ r_J \circ x = r_J x + \displaystyle \sum_{t=1}^{\abs{J}} \omega_{\left(J^{op}\right)_t} r_{J_{\abs{J}-t}} p_\Omega x l_{\left(J^{op}\right)_t} $$
		Now, by an application of \Cref{rieqn} we have,
		\begin{align*}
			r_I \circ x &= r_{i_1} \circ \left(r_J \circ x \right) \\
						&= r_{i_1} \left(r_J \circ x \right) + \omega_{i_1} p_\Omega \left(r_J \circ x\right) l_{i_1} \left(\t{By} \ \t{\Cref{rieqn}} \right) \\
						&= r_{i_1} \left(r_J x + \displaystyle \sum_{t=1}^{m} \omega_{\left(J^{op}\right)_t} r_{J_{m-t}} p_\Omega x l_{\left(J^{op}\right)_t}\right) + \omega_{i_1} p_\Omega \left(r_J x + \displaystyle \sum_{t=1}^{m} \omega_{\left(J^{op}\right)_t} r_{J_{m-t}} p_\Omega x l_{\left(J^{op}\right)_t}\right)l_{i_1} \\
						&= r_I x + \displaystyle \sum_{t=1}^{m} \omega_{\left(J^{op}\right)_t} r_{i_1} r_{J_{m-t}} p_\Omega x l_{\left(J^{op}\right)_t} + \omega_{i_1} \omega_{J^{op}} p_\Omega x l_{J^{op}} l_{i_1} \left(\t{Since} \ p_\Omega r_i =0 \ \t{for} \ i \in \Theta \right) \\
						&= r_I x + \displaystyle \sum_{t=1}^{m} \omega_{\left(I^{op}\right)_t} r_{I_{m+1-t}} p_\Omega x l_{\left(I^{op}\right)_t} + \omega_{I^{op}} p_\Omega x l_{I^{op}} \\
						&= r_I x + \displaystyle \sum_{t=1}^{m+1} \omega_{\left(I^{op}\right)_t} r_{I_{m+1-t}} p_\Omega x l_{\left(I^{op}\right)_t}
		\end{align*}
		Thus, the result is true for all sequences of length $m+1$. Hence, by the principle of mathematical induction, the result is true for all sequences in $\Lambda^*$.
		
		\item [(v)] Follows from part (iv) of the proposition by taking adjoint.
	\end{itemize}
\end{proof}

Consider the following vector state on $B(\mcal F(H))$, $$ x \longmapsto \lab x \Omega, \Omega \rab, x \in B(\mcal F(H)) . $$
Let $\phi$ be its restriction on $P(P_\omega)$. We collect some useful properties of $\phi$ in the next result.

\begin{prop}\label{formulaprop2}
	Let $x \in P(P_\omega)$ and $I, J \in \Lambda$. Then :
	\begin{itemize}
		\item [(i)] $\phi (x \circ r_J^*) = \omega_J \lab x \Omega, r_J \Omega \rab = \omega_J \phi(r_J^* \circ x) $.
		
		\item [(ii)] $\phi (r_J \circ x) = \omega_J \lab x r_J \Omega, \Omega = \omega_J \phi(x \circ r_J) $.
	\end{itemize}
\end{prop}
\begin{proof}
	\begin{itemize}
		\item [(i)] If $J=()$, is the empty tuple, we have nothing to prove. Suppose $J \in \Theta^k$ for some $k \in \N$, let $J = j_{i_1} j_{i_2} \cdots j_{i_k}$. Since $r_i^* (\Omega) = 0$ for all $i \in \Theta$, using \Cref{formulaprop}(v) we get that,
		\[ \left(x \circ r_J^* \right)\Omega = \left(x r_J^*\right)\Omega + \sum_{t=1}^{k} \omega_{J^{op}_t} (l_{J^{op}_t}^* x p_\Omega r_{I_{k-t}}^*) \Omega = \omega_J (l_{J^{op}}^* x)\Omega  \]
		Therefore,
		\[ \phi(x \circ r_J^*) = \lab (x \circ r_J^*)\Omega , \Omega \rab = \omega_J \lab (l_{J^{op}}^* x)\Omega, \Omega \rab = \omega_J \lab x \Omega, l_{J^{op}} \Omega \rab . \]
		Since, $l_{J^{op}}\Omega = r_J \Omega$, it follows that
		\[\phi(x \circ r_J^*) = \omega_J \lab x \Omega, r_J \Omega \rab = \omega_J \lab (r_J^* x)\Omega, \Omega \rab = \omega_J \lab (r_J^* \circ x)\Omega, \Omega \rab = \omega_J \phi(r_J^* \circ x) .  \]
		The third equality follows from \Cref{formulaprop}(ii).
		
		\item [(ii)] Follows from part (i) by taking adjoints.
	\end{itemize}
\end{proof}

\begin{lem}\label{rlem}
	Let $x$ be an element in $P(P_\omega)$ which commutes with every element of $F(P_\omega)$. Then $\lab x r_J \Omega, \Omega \rab = 0 $ for all $J \in \Lambda^*$. Further, if $I, J \in \Lambda $ are of same length, then $r_I^* x r_J = \delta_{I,J} x$.
\end{lem}

\begin{proof}
	For $J \in \Lambda^*$, we have,  \begin{equation}\label{eqn1}
		\lab x r_J \Omega, \Omega \rab = \phi (x \circ r_J) = \phi (r_J \circ x) = \omega_J \lab x r_J \Omega, \Omega \rab
	\end{equation} 
	We have the second equality because $x$ commute with elements of $F(P_\omega)$  and the last equality follows from \Cref{formulaprop2}(ii). Now since, $\abs{J} \geq 1$, we have that $\omega_J \neq 1$. Thus, it follows from \Cref{eqn1} that $\lab x r_J \Omega, \Omega \rab = 0 $.
	
	Further, if $I, J \in \Lambda$ are of same length, then we have,
	\begin{align*}
		r_I^* x r_J &= r_I^* \circ x \circ r_J \ \ \t{(By \Cref{formulaprop}(iii))} 	\\
		&= x \circ r_I^* \circ r_J \ \ \t{(Since $x$ is in center of $P(P_\omega)$)} \\
		&= x \circ r_I^* r_J \ \ \t{(By \Cref{formulaprop}(ii))} \\
		&= \delta_{I,J} x
	\end{align*} 
\end{proof} 

\begin{thm}\label{centerthm}
	The relative commutant of the inclusion $F(P_\omega) \subseteq P(P_\omega)$ is trivial.
\end{thm}

\begin{proof}
It suffices to show that given any element $x \in P(P_\omega)$ which commutes with every element of $F(P_\omega)$ satisfy $\lab x e_I, e_J \rab  = \delta_{I,J} \lab x \Omega, \Omega \rab $ for all $I, J \in \Lambda $. Let $x$ be an element in the relative commutant and let $I, J \in \Lambda$, say, $I = i_1 i_2 \cdots i_k$ and $J = j_1 j_2 \cdots j_m $

\ul{\textit{Case (i)}}: Suppose $I$ and $J$ are sequences of different lengths. Without loss of generality we assume $\abs{I} > \abs{J} $. Then $I^{op} = i_k \cdots i_1$. Let $I' = i_1 \cdots i_{k-m+1} $ and $I'' = i_{k-m} \cdots i_1$, so that $I^{op} = I' I''$. Since $k > m $, $\abs{I''} \geq 1 $, and hence, by \Cref{rlem}, we get $\lab x r{I''} \Omega, \Omega \rab = 0 $. Thus,
\begin{align*}
	\lab x e_I, e_J \rab &= \lab (x r_{I^{op}})\Omega, r_{J^{op}} \Omega \rab \\
						&= \lab (r_{J^{op}}^* x r_{I^{op}}) \Omega, \Omega \rab \\
						&= \lab (r_{J^{op}}^* x r_{I'} r_{I''}) \Omega, \Omega \rab \ \ \t{(Since $ I^{op} = I' I'' $ )} \\
						&= \delta_{J^{op}, I'} \lab x r_{I''} \Omega, \Omega \rab \ \ \t{(Since $\abs{I'} = \abs{J} = m$, by \Cref{rlem}, $ r_{J^{op}}^* x r_{I'} = \delta_{J^{op}, I'} x $ )} \\
						&= 0
\end{align*}

\ul{\textit{Case (ii)}} : Suppose $I$ and $ J $ are of same length. An application of \Cref{rlem} reveals that,
\[ \lab x e_I, e_J \rab = \lab (r_{J^{op}}^* x r_{I^{op}}) \Omega, \Omega \rab = \delta_{I^{op, J^{op}}} \lab x \Omega, \Omega \rab = \delta_{I,J} \lab x \Omega, \Omega \rab \]
This completes the proof.
\end{proof}

\begin{cor}
	The center of $P(P_\omega)$ is trivial.
\end{cor}
\begin{proof}
	Follows easily from \Cref{centerthm}.
\end{proof}

\subsection{Product in the Peripheral Poisson boundary }\ 	 

We aim to provide a description of the modified `Choi-Effros' product in the peripheral Poisson boundary $P(P_\omega)$. The following lemma will be useful for our purposes.
 
\begin{lem}\label{productlem}
	Let $a \in F(P_\omega)$. Then $a \circ x_\lambda = a x_\lambda$.
\end{lem}

\begin{proof}
 We have $ a \circ x_\lambda  = \t{SOT}-\us{n \to \infty}{\t{lim}} \lambda^{-n} P_\omega^n (a x_\lambda) $. For each $I \in \Lambda$, we get,
 \begin{align*}
 	P_\omega (a x_\lambda)(e_I) &= \us{i \in \Theta}{\sum} \omega_i l_i^* a x_\lambda l_i (e_I) \\ 
 	&= \us{i \in \Theta}{\sum} \omega_i l_i^* a x_\lambda (e_{iI}) \\
 	&=  \lambda^{\abs{I}+1} \us{i \in \Theta}{\sum} \omega_i l_i^* a (e_{iI}) \\
 	&=  \lambda^{\abs{I}+1} \us{i \in \Theta}{\sum} \omega_i l_i^* a l_i (e_I) \\
 	&= \lambda^{\abs{I}+1} a (e_I) \ \ \t{(Since $a \in F(P_\omega)$)} \\
 	&= \lambda a x_\lambda (e_I)
 \end{align*}
 Hence, $P_\omega (a x_\lambda) = \lambda a x_\lambda$. Similarly, we get $P_\omega^n (a x_\lambda) = \lambda^n a x_\lambda$. Thus, $a \circ x_\lambda = a x_\lambda $.
\end{proof}

Given a ucp map $P : M \to M$, recall that multiplicative domain of $ P $ consist of those elements of $ M $ such that $P(m^* m) = P(m)^* P(m)$ and $P(m m^*) = P(m) P(m)^*$. If $ m $ lies in the multiplicative domain of $P$ then, $P( m a) = P(m) P(a)$ and $P( b m ) = P(b) P(m)$ for all $a, b \in M$.

\begin{prop}
	Let $\lambda \in \mathbbm{T} \setminus \left\{1 \right\}$. Then,
	\[ E_\lambda (P_{\omega}) = \left\{ax_\lambda : a \in F(P_\omega)\right\} =\left\{x_\lambda a : a \in F(P_\omega)\right\} . \]
\end{prop}

\begin{proof}
	The first equality follows easily from \Cref{productlem} and \cite[Corollary 2.10]{BKT}. To establish the second equality, we first show that $x_\lambda$ lies in the multiplicative domain of $P_\omega$. Observe that, 
	\[P_\omega(x_\lambda)^* P_\omega(x_\lambda) = \ol{\lambda} \lambda x_\lambda^* x_\lambda = 1 = P_\omega (x_\lambda^* x_\lambda) \]
	Similarly, we get $P_\omega (x_\lambda) P_\omega (x_\lambda)^* = P_\omega (x_\lambda x_\lambda^*)$. Therefore, $x_\lambda$ lies in the multiplicative domain of $P_\omega$. Let $a \in F(P_\omega)$. Then, \begin{align*}
		P_\omega (x_\lambda a) &= P_\omega(x_\lambda) P_\omega (a) \ \ \t{(Since $x_\lambda$ is in multiplicative domain of $P_\omega$)} \\ 
		&= \lambda x_\lambda a . 
	\end{align*}
	So, $\left\{x_\lambda a : a \in F(P_\omega)\right\} \subseteq E_\lambda(P_\omega)	$.
	
	Conversely, let $T \in E_\lambda (P_\omega)$. Let $a = x_\lambda^* T$. To show that $a \in F(P_\omega)$. Now,
	\[ P_\omega(a) = P_\omega (x_\lambda^* T) = P_\omega (x_\lambda)^* P_\omega (T) = \ol{\lambda} \lambda x_\lambda^* T = a  . \]
	This concludes the proposition.
\end{proof}

The following lemma will be useful to describe the product of two elements of $E (P_\omega)$ in the peripheral Poisson boundary.

\begin{lem}
	For $a , b \in F(P_\omega)$ and $\lambda, \mu \in \mathbbm{T} \setminus \left\{1 \right\} $, we have, $$ P_\omega((x_\lambda a)(b x_\mu) ) = (\lambda \mu) x_\lambda P_\omega (ab) x_\mu . $$
\end{lem}
\begin{proof}
	For $I,J \in \Lambda$, we obtain,
	\begin{align*}
		\lab P_\omega (x_\lambda a b x_\mu)(e_I), e_J \rab &= \mu^{\abs{I}+1} \us{i \in \Theta}{\sum} \omega_i \lab x_\lambda a b (e_{iI}), e_{iJ} \rab \\
		&= \mu^{\abs{I}+1} \us{i \in \Theta}{\sum} \omega_i \lab a b (e_{iI}), x_\lambda^* (e_{iJ}) \rab \\
		&= \mu^{\abs{I}+1} \us{i \in \Theta}{\sum} \omega_i \lab a b (e_{iI}), \ol{\lambda}^{\abs{J}+1} e_{iJ} \rab \\
		&= \mu^{\abs{I}+1} \lambda^{\abs{J}+1} \us{i \in \Theta}{\sum} \omega_i \lab a b (e_{iI}),  e_{iJ} \rab \\
		&= \mu^{\abs{I}+1} \lambda^{\abs{J}+1} \lab P_\omega(ab)(e_I), e_J \rab \\
		&= (\lambda \mu) \, \lab (P_\omega(ab) x_\mu )(e_I), \, x_{\ol \lambda} (e_J) \rab \\
		&= (\lambda \mu) \, \lab (x_\lambda P_\omega(ab) x_\mu )(e_I), \, e_J \rab
	\end{align*}
	Thus, we have the desired result. 
\end{proof}

Let $x_\lambda a$ and $b x_\mu$ be two elements in $E(P_\omega)$, where $a,b \in F(P_\omega)$ and $\lambda, \mu \in \mathbbm{T} \setminus \left\{1 \right\} $. Making use of the preceding lemma, we get that, $P_\omega (x_\lambda (ab) x_\mu) = (\lambda \mu) \, x_\lambda P_\omega(ab) x_\mu $. Similarly, we get $P_\omega^2 (x_\lambda (ab) x_\mu) = (\lambda \mu)^2 \, x_\lambda P_\omega^2(ab) x_\mu$. Iterating this process, we obtain, $P_\omega^n \, (x_\lambda (ab) x_\mu) = (\lambda \mu)^n x_\lambda P_\omega^n(ab) x_\mu$. Thus, we get that,
\begin{align*}
	(x_\lambda a) \circ (b x_\mu) &= \t{SOT}-\us{n \to \infty}{\t{lim}} (\lambda \mu)^{-n} P_\omega^n \, (x_\lambda (ab) x_\mu) \\
	&=  \t{SOT}-\us{n \to \infty}{\t{lim}} (\lambda \mu)^{-n} (\lambda \mu)^n x_\lambda P_\omega^n (ab) x_\mu \\
	&= x_\lambda (a \circ b) x_\mu .
\end{align*}

\subsection{Peripheral Poisson boundary as a Hilbert C*-bimodule over the Poisson boundary}\

We make use of the injectivity of $F(P_\omega)$ to build a conditional expectation from $P(P_\omega)$ to $F(P_\omega)$ and further use it to induce a bimodule structure on $P(P_\omega) $.

\begin{rem}
	In \cite{BKT}, for any normal ucp $\tau : M \to M $ and for each $\lambda \in \mathbbm{T}$, $E_\lambda (\tau)$ has been equipped with natural left and right actions of $F(\tau)$, and with a $F(\tau)$-valued inner product (where $\lab x, y \rab = x^* \circ y $ for $x, y \in E_\lambda (\tau)$), making it a Hilbert C*-bimodule over $F(\tau)$. This approach is not going to work in general if we want to make $P(\tau)$ a pre-Hilbert C*-bimodule over $F(\tau)$. This is because if $x \in E_\lambda (\tau)$ and $y \in E_\mu (\tau)$, for $\lambda \neq \mu$,  then, $$x^* \circ y \in E_{\ol{\lambda} \mu} (\tau) \neq F(\tau) $$
	In the subsequent results, we establish $P(P_\omega)$ as a pre-Hilbert C*-bimodule over $F(P_\omega)$
\end{rem}

\begin{prop}\label{CEppb}
	There is a conditional expectation $E : P(P_\omega) \to F(P_\omega) $.
\end{prop}

\begin{proof}
	Using \cite[Remark 15]{BBDR}, we get that $F(P_\omega)$ is injective. Consider the identity map $\psi : F(P_\omega) \to F(P_\omega)$. Now $F(P_\omega)$ being a C*-subalgebra of $P(P_\omega)$ and $F(P_\omega)$ being injective, the map $\psi : F(P_\omega) \to F(P_\omega)$ extends to a completely positive idempotent contraction $E : P(P_\omega) \to F(P_\omega) $. This concludes the proposition.  
\end{proof}

\begin{thm}
	$P(P_\omega)$ becomes a pre-Hilbert C*-bimodule over the Poisson boundary $F(P_\omega) $.
\end{thm}

\begin{proof}
	$F(P_\omega)$ being a C*-subalgebra of $P(P_\omega)$, has natural left and right actions on $P(P_\omega)$. Now using the conditional expectation $E : P(_\omega) \to F(P_\omega)$ obtained in \Cref{CEppb}, we obtain a $F(P_\omega)$-valued inner product, \[\lab x, y \rab = E(x \circ y^*) \ \ \t{for} \ x, y \in P(P_\omega) . \]
	Since $E$ is a conditional expectation, it is easy to verify that $P(P_\omega)$ becomes a pre-Hilbert C*-bimodule over $F(P_\omega)$. This completes the proof.
\end{proof}

\begin{rem}
	We believe that the condition expectation obtained in \Cref{CEppb} is never of finite index.
\end{rem}

\begin{rem}
	Given a unitary $U$ on $H$, consider the corresponding second quantization $\Gamma_U$ on $\mcal F(H)$ and the associated automorphism $\widetilde \Gamma_U$ of $B(\mcal F(H))$. In \cite{BBDR}, the question whether the restriction of $\widetilde \Gamma_U$ on $F(P_\omega)$ becomes an automorphism of $F(P_\omega)$ was addressed. It was shown that when the sequence of positive real numbers $\omega = \left(\omega_i\right)_{i \in \Theta}$ with $\us{i \in \Theta} \sum \omega_i = 1$ is not constant the question is not true. When the Hilbert space $H$ is finite dimensional (say, $\t{dim}(H) = n > 1$) and $\omega$ is the constant sequence $\frac{1}{n}$ (that is, $\omega_1 = \omega_2 = \cdots = \omega_n = \frac{1}{n}$) then it was proved in \cite{BBDR} that, the restriction of $\widetilde \Gamma_U$ on $F(P_\omega)$ is an automorphism of $F(P_\omega)$. An easy application of the same arguments will reveal that $\widetilde \Gamma_U$ is also an automorphism of $P(P_\omega) $.
\end{rem}  		 

\begin{rem}
	When $\t{dim}(H) = 1$, $P(P_\omega)$ turns out to be the peripheral Poisson boundary of $\lambda$-Toeplitz operators. 
	In \cite{BKT}, the peripheral Poisson boundary of $\lambda$-Toeplitz operators was discussed. It was shown that the corresponding peripheral Poisson boundary is not a von Neumann algebra. We also believe that for $\t{dim}(H) > 1$, $P(P_\omega)$ will also not be a von Neumann algebra.  
\end{rem}



\end{document}